\documentclass[a4paper]{amsart}
\usepackage{graphicx}
\usepackage[bookmarksnumbered,colorlinks,plainpages]{hyperref}

\usepackage{amsthm}
\usepackage[all,2cell]{xy}
\usepackage{amsfonts}
\usepackage{amsmath}
\usepackage{amssymb}
\usepackage{xypic,leqno,amscd,amssymb,pstricks,latexsym,amsbsy,xypic,mathrsfs,verbatim}
\usepackage{multirow}
\usepackage{booktabs} 
\usepackage{makecell} 

\UseAllTwocells \SilentMatrices
\newtheorem{thm}{Theorem}[section]

\newtheorem{cor}[thm]{Corollary}
\newtheorem{lem}[thm]{Lemma}
\newtheorem{exm}[thm]{Example}

\newtheorem{prop}[thm]{Proposition}
\theoremstyle{definition}

\theoremstyle{remark}

\numberwithin{equation}{section}

\def\Aut{{\rm Aut}}

\def\coh{{\rm coh}\mbox{-}}

\def\bbY{{\mathbb Y}}
\def\bbX{{\mathbb X}}
\def\bbL{{\mathbb L}}

\def\bbE{{\mathbb E}}
\def\bbZ{{\mathbb Z}}
\def\bfq{{\mathbf{q}}}
\def\bfp{{\mathbf{p}}}

\renewcommand{\mod}{\operatorname{mod}\nolimits}

\begin{document}
\title[Equivariant approach to weighted projective curves]{Equivariant approach to weighted projective curves}

\author[Qiang Dong, Shiquan Ruan, Hongxia Zhang] {Qiang Dong, Shiquan Ruan, Hongxia Zhang$^*$}

\address{Qiang Dong, Shiquan Ruan and Hongxia Zhang;\; School of Mathematical Sciences, Xiamen University, Xiamen, 361005, Fujian, P.R. China.}
\email{dongqiang@stu.xmu.edu.cn, sqruan@xmu.edu.cn, hxzhangxmu@163.com}

\thanks{$^*$ the corresponding author}
\subjclass[2010]{16W50, 14A22, 30F10, 14F05}
\date{\today}
\keywords{Weighted projective curve, weighted projective line, group action, equivariantization, coherent sheaf}

\begin{abstract}
We investigate group actions on the category of coherent sheaves over weighted projective lines. We show that the equivariant category with respect to certain finite group action is equivalent to the category of coherent sheaves over a weighted projective curve, where the genus of the underlying smooth projective curve and the weight data are given by explicit formulas. Moreover, we obtain a trichotomy result for all such equivariant relations.

This equivariant approach provides an effective way to investigate smooth projective curves via weighted projective lines. As an application, we give a description for the category of coherent sheaves over a hyperelliptic curve of arbitrary genus. Links to smooth projective curves of genus 2 and also to Arnold's exceptional unimodal singularities are also established.
\end{abstract}

\maketitle

\section{Introduction}

Weighted projective lines and their coherent sheaves categories are introduced by Geigle and Lenzing in \cite{GeigleLenzing1987}, in order to give a geometric realization of canonical algebras in the sense of Ringel \cite{Ringel1984}. In \cite{Lenzing2017} Lenzing showed that any weighted projective line (with little exceptions) can be viewed as an orbifold quotient of a compact Riemann surface with respect to a finite group action, and its category of coherent sheaves is equivalent to the category of equivariant sheaves on the associated Riemann surface. Accordingly, the study of weighted projective lines has a high contact surface with many mathematical subjects. Among the many related subjects we mention Lie theory \cite{Crawley-Boevey2010,DengRuanXiao2020,DouJiangXiao2012,Schiffmann2004}, automorphic
forms \cite{Dolgacev1975,KussinMeltzer2002,Meltzer1995,Montesinos1987,Serre1955}, orbifolds \cite{Milnor1975,MiyachiYekutieli2001,Neumann1977,RingelSchmidmeier2008},
and singularities \cite{EbelingPloog2010,Hubner1989,Hubner1996,Lenzing1994,Lenzing1998,Lenzing2011}. This latter aspect, in particular, includes
the aspects of Arnold's strange duality \cite{EbelingPloog2010,EbelingTakahashi2013} and homological mirror symmetry \cite{Ebeling2003,Scerbak1978}.

Let $\mathbf{k}$ be an algebraically closed field of characteristic zero. Let $\mathbb{P}_{\mathbf k}^1(\mathbf{p};{\boldsymbol\lambda})$ denote the weighted projective line of weight type ${\bf p}=(p_1, p_2, \cdots, p_t)$ and parameter sequence ${\boldsymbol\lambda}=(\lambda_1, \lambda_2, \cdots, \lambda_t)$, where $p_i$'s are integers at least 2, and $\lambda_i$'s are pairwise distinct points on the projective line $\mathbb{P}_{\mathbf k}^1$, normalized as $\lambda_1=\infty$, $\lambda_2=0$ and $\lambda_3=1$.
According to the sign of the Euler characteristic, the weighted projective lines are divided into \emph{domestic, tubular} and \emph{wild} types.
We recall that the \emph{Picard group} of $\mathbb{P}_{\mathbf k}^1(\mathbf{p};{\boldsymbol\lambda})$ is naturally isomorphic to the rank-one abelian group $\bbL({\bf p})$ on generators $\vec{x}_1, \vec{x}_2, \cdots, \vec{x}_t$ subject to the relations $p_1\vec{x}_1=p_2\vec{x}_2=\cdots = p_t\vec{x}_t:=\vec{c}$, where $\vec{c}$ is called the \emph{canonical element} and $\vec{\omega}:=(t-2)\vec{c}-\sum_{i=1}^t\vec{x}_i$ is called the \emph{dualizing element} of $\bbL({\bf p})$.

Due to the work of Lenzing and Meltzer \cite{LenzingMelter2000}, the group $\Aut(\coh\mathbb{P}_{\mathbf k}^1(\mathbf{p};{\boldsymbol\lambda}))$ of automorphisms on the category of coherent sheaves over $\mathbb{P}_{\mathbf k}^1(\mathbf{p};{\boldsymbol\lambda})$ has an explicit description (see \eqref{automorphism group} below), which contains two fundamental parts: the \emph{geometric automorphisms} $\Aut(\mathbb{P}_{\mathbf k}^1(\mathbf{p};{\boldsymbol\lambda}))$ and the \emph{degree-shift automorphisms} $\bbL({\bf p})$.

The geometric automorphisms consist of automorphisms of $\mathbb{P}_{\mathbf k}^1$ that preserve weights, hence $\Aut(\mathbb{P}_{\mathbf k}^1(\mathbf{p};{\boldsymbol\lambda}))$ is of finite when $t\geq 3$; in particular, it is trivial if $p_i$'s are pairwise distinct. The associated \emph{geometric action} on the category $\coh\mathbb{P}_{\mathbf k}^1(\mathbf{p};{\boldsymbol\lambda})$ always fixes the structure sheaf and permute almost all the ordinary simple sheaves, while the \emph{degree-shift action} always moves the structure sheaf and fix all the ordinary simple sheaves. On the other hand, the \emph{geometric action} always yields weighted projective lines as quotient stacks, while the \emph{degree-shift action} will achieve covering curves, including elliptic curves and other weighted projective curves (see Theorem \ref{trichotomy} below).

In this paper we will focus on the degree-shift actions. Denote by $t\bbL(\bf{p})$ the torsion subgroup of $\bbL({\bf p})$. For any subgroup $H$ of $t\bbL(\bf{p})$, the associated \emph{equivariant category} will be denoted by $\big({\rm coh}\mbox{-}\mathbb{P}_{\mathbf k}^1(\mathbf{p};{\boldsymbol\lambda})\big)^H$. The \emph{equivariantization} with respect to certain $H$ has been considered in literature. In \cite{CCZ}, the authors used $H=\mathbb{Z}\vec{\omega}$ to establish the relationship between tubular weighted projective lines and elliptic curves; compare \cite{CCR,GeigleLenzing1987,Lenzing2004,Polishchuk2006}. Explicit equivariant relations between distinct tubular weighted projective lines are established in \cite{CC} via $H=\mathbb{Z}(2\vec{\omega})$ or $\mathbb{Z}(3\vec{\omega})$. The authors in \cite{CLRZ} introduced the notion of Cylic/Klein type for $H$ to characterize the equivariant relations between weighted projective lines of arbitrary types. In this paper, we will consider equivariantization with respect to arbitrary subgroup $H$ of $t\bbL(\bf{p})$. The main result is as follows:

\begin{thm}\label{trichotomy}
Let $\mathbb{P}_{\mathbf k}^1(\mathbf{p};{\boldsymbol\lambda})$ be a weighted projective line and $H$ be a finite subgroup of $\mathbb{L}(\bf{p})$. Then
\begin{equation}\label{main equivalence}({\rm coh}\mbox{-}\mathbb{P}_{\mathbf k}^1(\mathbf{p}; {\boldsymbol\lambda}))^{H}\stackrel{\sim}\longrightarrow {\rm coh}\mbox{-}\mathbf{Y}
 \end{equation} for some weighted projective curve $\mathbf{Y}=\mathbb{Y}(\mathbf{q};{\boldsymbol\mu})$. More precisely,
\begin{itemize}
   \item[(1)] $\mathbb{P}_{\mathbf k}^1(\mathbf{p};{\boldsymbol\lambda})$ is of domestic type if and only if $\mathbf{Y}$ is a weighted projective line of domestic type;
   \item[(2)] $\mathbb{P}_{\mathbf k}^1(\mathbf{p};{\boldsymbol\lambda})$ is of tubular type if and only if $\mathbf{Y}$ is a weighted projective line of tubular type or a smooth elliptic curve, according to $H\not\supseteq\mathbb{Z}\vec{\omega}$ or $H\supseteq \mathbb{Z}\vec{\omega}$ respectively;
   \item[(3)] $\mathbb{P}_{\mathbf k}^1(\mathbf{p};{\boldsymbol\lambda})$ is of wild type if and only if for all the other $\mathbf{Y}$'s.
 \end{itemize}
 Moreover, such equivalences \eqref{main equivalence} for $\mathbb{P}_{\mathbf k}^1(\mathbf{p};{\boldsymbol\lambda})$ being domestic type or tubular type are classified in Tables
 \ref{table 1} and \ref{table 2}, where $\bbE(\lambda)$ is the smooth elliptic curve determined by the equation {$y^2z=x(x-z)(x-\lambda z)$}, $\omega$ satisfies $\omega^2-\omega+1=0$, $\Gamma(\lambda)=\{\lambda, \frac{1}{\lambda},1-\lambda,\frac{1}{1-\lambda},\frac{\lambda}{\lambda-1}, \frac{\lambda-1}{\lambda}\}$ and $i,j,k,l$ are pairwise distinct integers.

\begin{table}[h]
\caption{Equivariant relations for domestic type}
\begin{tabular}{|c|c|c|c|c|}
  \hline
 \makecell*[c]{$\mathbb{P}_{\mathbf k}^1(\mathbf{p};{\boldsymbol\lambda})$}&$H$&$\mathbf{Y}=\mathbb{Y}(\mathbf{q};{\boldsymbol\mu})$\\
\hline
 \makecell*[c]{$\mathbb{P}_{\mathbf k}^1(2,3,3)$}&$\langle\vec{x}_2-\vec{x}_3\rangle$&$\mathbb{P}_{\mathbf k}^1(2,2,2)$\\
  \hline
 \makecell*[c]{$\mathbb{P}_{\mathbf k}^1(2,3,4)$}&$\langle\vec{x}_1-2\vec{x}_3\rangle$&$\mathbb{P}_{\mathbf k}^1(2,3,3)$\\
  \hline
  \makecell*[c]{$\mathbb{P}_{\mathbf k}^1(2,2,n)$}&$\langle\vec{x}_1-\vec{x}_2\rangle$&$\mathbb{P}_{\mathbf k}^1(n,n)$\\
  \hline
 \multirow{3}*{$\mathbb{P}_{\mathbf k}^1(2,2,2n)$}&$\langle \vec{x}_1-n\vec{x}_3\rangle$&\makecell*[c]{$\mathbb{P}_{\mathbf k}^1(2,2,n)$}\\
\cline{2-3}
&$\langle \vec{x}_1-\vec{x}_2,\,\vec{x}_1-n\vec{x}_3\rangle$&\makecell*[c]{$\mathbb{P}_{\mathbf k}^1(n,n)$}\\
\hline
\makecell*[c]{$\mathbb{P}_{\mathbf k}^1(np_1,np_2)$}&$\langle kp_1\vec{x}_1-kp_2\vec{x}_2\rangle \quad (k|n)$&$\mathbb{P}_{\mathbf k}^1(kp_1,kp_2)$\\
  \hline
 \end{tabular}
 \label{table 1}
\end{table}

\begin{table}[h]
\caption{Equivariant relations for tubular type}

\begin{tabular}{|c|c|c|}
\hline
\makecell*[c]{$\mathbb{P}_{\mathbf k}^1(\mathbf{p};{\boldsymbol\lambda})$}&$H$&$\mathbf{Y}=\mathbb{Y}(\mathbf{q};{\boldsymbol\mu})$\\
\hline
\multirow{11}*{$\mathbb{P}_{\mathbf k}^1(2,2,2,2;\lambda)$}&$\langle \vec{x}_1-\vec{x}_2\rangle$ \textup{or} $\langle \vec{x}_3-\vec{x}_4\rangle$ &\makecell*[c]{$\mathbb{P}_{\mathbf k}^1(2,2,2,2;\mu)$;\;\;$\mu\in\Gamma\big((\frac{\sqrt{\lambda}+1}{\sqrt{\lambda}-1})^2\big)$}\\
\cline{2-3}
&$\langle \vec{x}_1-\vec{x}_3\rangle$ \textup{or} $\langle \vec{x}_2-\vec{x}_4\rangle$&\makecell*[c]{$\mathbb{P}_{\mathbf k}^1(2,2,2,2;\mu)$;\;\;$\mu\in\Gamma\big((\frac{\sqrt{1-\lambda}+1}{\sqrt{1-\lambda}-1})^2\big)$}\\
\cline{2-3}
&$\langle \vec{x}_1-\vec{x}_4\rangle$ \textup{or} $\langle \vec{x}_2-\vec{x}_3\rangle$&\makecell*[c]{$\mathbb{P}_{\mathbf k}^1(2,2,2,2;\mu)$;\;\;$\mu\in\Gamma\big((\frac{\sqrt{\lambda}+\sqrt{\lambda-1}}{\sqrt{\lambda}-\sqrt{\lambda-1}})^2\big)$}\\
\cline{2-3}
&$\langle \vec{x}_i-\vec{x}_j,\vec{x}_i-\vec{x}_k\rangle$&\makecell*[c]{$\mathbb{P}_{\mathbf k}^1(2,2,2,2;\mu)$;\;\;$\mu\in\Gamma(\lambda)$}\\
\cline{2-3}
&\makecell*[c]{$\mathbb{Z}\vec{\omega}$}&\multirow{4}*{$\mathbb{E}(\mu)$;\;\;$\mu\in\Gamma(\lambda)$}\\
\cline{2-2}
&\makecell*[c]{$\langle \vec{x}_i-\vec{x}_j,\,\vec{x}_k-\vec{x}_l\rangle$}&\\
\cline{2-2}
&\makecell*[c]{$t\bbL(2,2,2,2)$}&\\
\hline
\multirow{8}*{$\mathbb{P}_{\mathbf k}^1(4,4,2)$}&$\langle 2\vec{x}_1-\vec{x}_3\rangle$ \textup{or} $\langle 2\vec{x}_2-\vec{x}_3\rangle$&\makecell*[c]{$\mathbb{P}_{\mathbf k}^1(4,4,2)$}\\
\cline{2-3}
&\makecell*[c]{$\mathbb{Z}(2\vec{\omega})$}&\multirow{4}*{$\mathbb{P}_{\mathbf k}^1(2,2,2,2;\mu)$;\;\;$\mu\in\Gamma(-1)$}\\
\cline{2-2}
&\makecell*[c]{$\langle \vec{x}_1-\vec{x}_2\rangle$}&\\
\cline{2-2}
&\makecell*[c]{$\langle 2\vec{x}_1-\vec{x}_3,\,2\vec{x}_1-2\vec{x}_2\rangle$}&\\
\cline{2-3}
&\makecell*[c]{$\mathbb{Z}\vec{\omega}$}&\multirow{3}*{$\mathbb{E}(\mu)$;\;\;$\mu\in\Gamma(-1)$}\\
\cline{2-2}
&\makecell*[c]{$t\bbL(4,4,2)$}&\\
\hline
\multirow{4}*{$\mathbb{P}_{\mathbf k}^1(6,3,2)$}&$\mathbb{Z}(3\vec{\omega})$&\makecell*[c]{$\mathbb{P}_{\mathbf k}^1(3,3,3)$}\\
\cline{2-3}
&$\mathbb{Z}(2\vec{\omega})$&\makecell*[c]{$\mathbb{P}_{\mathbf k}^1(2,2,2,2;\mu)$;\;\;$\mu\in\Gamma(\omega)$}\\
\cline{2-3}
&$\mathbb{Z}\vec{\omega}=t\bbL(6,3,2)$&\makecell*[c]{$\mathbb{E}(\mu)$;\;\;$\mu\in\Gamma(\omega)$}\\
\hline
\multirow{4}*{$\mathbb{P}_{\mathbf k}^1(3,3,3)$}&$\langle \vec{x}_i-\vec{x}_j\rangle$&\makecell*[c]{$\mathbb{P}_{\mathbf k}^1(3,3,3)$}\\
\cline{2-3}
&\makecell*[c]{$\mathbb{Z}\vec{\omega}$}&\makecell*[c]{\multirow{3}*{$\mathbb{E}(\mu)$;\;\;$\mu\in\Gamma(\omega)$}}\\
\cline{2-2}
&\makecell*[c]{$t\bbL(3,3,3)$}&\\
\hline
\end{tabular}
\label{table 2}
\end{table}\end{thm}

The classification result above relies on the following explicit formulas on the weight data $\mathbf{q}$ of the weighted projective curve $\mathbf{Y}=\mathbb{Y}(\mathbf{q};{\boldsymbol\mu})$ and the genus $g_{\bbY}$ of the underlying smooth projective curve $\bbY$.
\begin{prop}\label{change of weight}
Keep the notations in Theorem \ref{trichotomy}. Let ${\bf p}=(p_1, p_2, \cdots, p_t)$. Assume $|H|=n$ and $H$ has a set of generators
$\vec{y}_i=\sum_{j=1}^t a_{ij}\vec{x}_j+a_{i}\vec{c}$ for $1\leq i\leq s$.
Denote by ${\rm g.c.d.}(a_{1j}, a_{2j}, \cdots, a_{sj}, p_j)=d_j$ and set $n_j=\frac{n d_j}{p_j}$ for each $1\leq j\leq t$. Then
$$g_{\bbY}=\frac{n}{2}\sum_{j=1}^t(1- \frac{d_j}{p_j})-n+1 \quad {\text{and}}\quad {\bf q}=(\underbrace{d_1,\cdots, d_1}_{n_1 \ \text{times}}, \underbrace{d_2,\cdots, d_2}_{n_2 \ \text{times}},\cdots, \underbrace{d_t,\cdots, d_t}_{n_t \ \text{times}}).$$
\end{prop}

The equivalence \eqref{main equivalence} provides a new way to investigate the category of coherent sheaves over smooth projective curves via weighted projective lines.

For any integer $g>0$, let $a_1, a_2,\cdots, a_g$ be pairwise distinct elements in ${\bf{k}}\setminus\{0,1\}$.
Let ${\mathbf{E}}_g$ be the hyperelliptic curve of genus $g$ determined by the following equation: $$z^2=(y^2-x^2)(y^2-a_1x^2)\cdots(y^2-a_gx^2).$$

\begin{prop}(Proposition \ref{hyperelliptic})
Keep the notations as above. Let ${\bf{p}}=(2,\cdots, 2)$ ($g+3$ times) and $\boldsymbol\lambda=(\infty,0,1,a_1,\cdots,a_g)$. Then there exists an equivalence
$$({\rm coh}\mbox{-}\mathbb{P}_{\mathbf k}^1(\mathbf{p}, {\boldsymbol\lambda}))^{H}\stackrel{\sim}\longrightarrow {\rm coh}\mbox{-}\mathbf{E}_g,$$
where $H=\langle \vec{x}_{1}-\vec{x}_{2},\, \delta_{g,even} \vec{x}_{2}+\sum_{i=3}^{g+3}(-1)^i \vec{x}_i \rangle\leq \mathbb{L}(\bf{p})$.
\end{prop}

Moreover, we also classify the equivariant relations \eqref{main equivalence} for the following cases:
\begin{itemize}
  \item[(i)] $\mathbf{Y}$ is a smooth projective curve of genus 2; (see Table \ref{smooth of genus 2});
  \item[(ii)] $\mathbb{P}_{\mathbf k}^1(\mathbf{p};{\boldsymbol\lambda})$'s are corresponding to Arnold's 14 exceptional unimodal singularities;
  (see Table \ref{Gorenstein parameter -1}).
\end{itemize}

The paper is organized as follows.
In Section 2, we recall the definition of a group action on a category, where a kind of group action on a length category is discussed. We recall the axiomatic description for the category of coherent sheaves over a weighted projective curve in Section 3, and show that the axioms are preserved under certain group actions. In the main part Section 4, we investigate the group action on the category of coherent sheaves over weighted projective lines, and prove Theorem \ref{trichotomy} and Proposition \ref{change of weight}. Several applications to smooth projective curves of higher genus are given in the last Section 5.

\section{Group action on a category}

In this section, we recall the definition of a group action on an abelian category. Moreover, an action on a tube is discussed in Proposition \ref{group action on cyclic quiver}, which plays an important role in this paper.

\subsection{Group action}

In this section we recall the notions of a finite group action on a small abelian category; for more details we refer to \cite{Chenxiaowu2017,Deligne1997,DrinfeldGelakiNikshychOstrik2010,ReitenRiedtmann1985}. Throughout the paper we always denote by $\mathcal{A}$ a small $\mathbf k$-linear abelian category.

Let $G$ be an arbitrary finite group. A $\mathbf k$-linear action of $G$ on $\mathcal{A}$ consists of the date $\{F_g,\varepsilon_{g,h}\,|\,g,h\in G\}$, where each $F_g\colon \mathcal{A}\rightarrow\mathcal{A}$ is a $\mathbf k$-linear auto-equivalence and each $\varepsilon_{g,h}\colon F_gF_h\rightarrow F_{gh}$ is a natural isomorphism such that a 2-cocycle condition holds, that is, $\varepsilon_{gh,k}\circ \varepsilon_{g,h}F_k=\varepsilon_{g,hk}\circ F_g\varepsilon_{h,k}$ for all $g,h,k\in G$.

Let $\{F_g,\varepsilon_{g,h}\,|\,g,h\in G\}$ be a $\mathbf k$-linear $G$-action on $\mathcal{A}$. A \emph{$G$-equivariant object} in $\mathcal{A}$ is a pair $(X,\alpha)$, where $X$ is an object in $\mathcal{A}$ and $\alpha$ assigns for each $g\in G$ an isomorphism $\alpha_g\colon X\rightarrow F_{g}X$ subject to the relations $$\alpha_{gg'}=(\varepsilon_{g,g'})_{X}\circ F_g(\alpha_{g'})\circ\alpha_g.$$
A morphism $f\colon (X,\alpha)\rightarrow (Y,\beta)$ between two $G$-equivariant objects is a morphism $f\colon X\rightarrow Y$ in $\mathcal{A}$ such that $\beta_{g}\circ f=F_{g}(f)\circ\alpha_{g}$ for all $g\in G$. This gives rise to the \emph{equivariant category} $\mathcal{A}^{G}$ of $G$-equivariant objects, which is again abelian. The process forming the equivariant category $\mathcal{A}^{G}$ is known as the \emph{equivariantization} of $\mathcal{A}$ with respect to the group action.

We recall that the \emph{forgetful functor} $U\colon \mathcal{A}^{G}\rightarrow \mathcal{A}$ is defined by $U(X,\alpha)=X$ and $U(f)=f$ for any $f\colon (X,\alpha)\rightarrow (Y,\beta)$. The \emph{induction functor} $F$ is defined as follows: for an object $X$ in $\mathcal{A}$, set $F(X)=(\bigoplus_{h\in G}F_{h}(X),\varepsilon)$, where for each $g\in G$, the isomorphism $\varepsilon_{g}\colon\bigoplus_{h\in G}F_{h}(X)\rightarrow F_g(\bigoplus_{h\in G}F_{h}(X))$ is diagonally induced by the isomorphism $(\varepsilon_{g,g^{-1}h})_{X}^{-1}\colon F_{h}(X) \rightarrow F_g(F_{g^{-1}h}(X))$. For each morphism $\theta$: $X\rightarrow Y$ in $\mathcal{A}$, $F(\theta)=\oplus_{h\in G}F_h(\theta)$: $F(X)\rightarrow F(Y)$.

According to Proposition 4.12, Corollary 4.13 and Lemma 4.14 in \cite{Zhou2015}, we know that both of the functors $F$ and $U$ are faithful and exact, and $(F,U)$, $(U,F)$ form adjoint pairs; compare \cite[Theorem 3.2]{ReitenRiedtmann1985}.

A given $\mathbf k$-linear action $\{F_g,\varepsilon_{g,h}\,|\,g,h\in G\}$ is \emph{strict} provided that each $F_g$ is an automorphism and each isomorphism $\varepsilon_{g,h}$ is the identity. Therefore, a strict action coincides with a group homomorphism from $G$ to the $\mathbf k$-linear automorphism group of $\mathcal{A}$. In this paper, we always assume the group action is strict.

\subsection{Serre duality}

A Hom-finite $\mathbf k$-linear abelian category $\mathcal{A}$ is said to satisfy \emph{Serre duality} if there exists an equivalence $\tau\colon \mathcal{A}\rightarrow\mathcal{A}$ with functorial $\mathbf k$-linear isomorphisms
\begin{align}\label{Serre duality} D\text{Ext}_{\mathcal{A}}^{1}(X,Y)\cong \text{Hom}_{\mathcal{A}}(Y,\tau X)
\end{align}
for all $X,Y$ in $\mathcal{A}$, where $D=\text{Hom}_{\mathbf k}(-,\mathbf k)$ denotes the standard $\mathbf k$-duality. The self-equivalence functor $\tau$ is called \emph{Serre functor} or \emph{Auslander-Reiten translation}. Note that a Serre functor is $\mathbf k$-linear and essentially unique provided it exists, this follows from Yoneda's lemma; see \cite[Remark 2.4]{Chenxiaowu2014}.

\subsection{Length category and uniserial category}
We recall from \cite{ChenxiaowuKrause2009} basic facts on the length categories and uniserial categories.

Let $\mathcal{A}$ be an abelian category. An object $X$ of $\mathcal{A}$ has \emph{finite length} if there exists a finite chain of subobjects $0=X_0\subseteq X_1\subseteq \cdots\subseteq X_{n-1}\subseteq X_n=X$ such that each quotient $X_i/X_{i-1}$ is a simple object. Such a chain is called a \emph{composition series} of $X$. A composition series is not necessarily unique but its length is an invariant of $X$ by the Jordan-H\"older theorem; it is called the \emph{length} of $X$ and is denoted by $l(X)$. Note that an object has finite length if and only if it is both artinian and noetherian (i.e. it satisfies the descending chain condition and ascending chain condition on subobjects, respectively).

According to the Krull-Remak-Schmidt theorem, every object of finite length decomposes essentially uniquely into a finite direct sum of indecomposable objects with local endomorphism rings.
The objects of finite length form a (Serre) subcategory of $\mathcal{A}$ which is denote by $\mathcal{A}_0$. The abelian category $\mathcal{A}$ is called a \emph{length category} if $\mathcal{A}=\mathcal{A}_0$.

Let $\mathcal{A}$ be a length category. An object $X$ is \emph{uniserial} if it has a unique composition series. Note that any non-zero uniserial object is indecomposable. Moreover, subobjects and quotient objects of uniserial objects are uniserial. The length category $\mathcal{A}$ is called \emph{uniserial category} provided that each indecomposable object is uniserial.

\subsection{Group action on a tube}
{For any integer $p\geq 1$, denote by $(\mathcal{T}_p,\tau)$ (or just simplified $\mathcal{T}_p$) the stable tube of rank $p$, that is, the Auslander-Reiten quiver of $\mathcal{T}_p$ is of form $\mathbb{ZA}_{\infty}/(\tau^p)$; see \cite{SimsonSkowronski2007}. We consider a finite abelian group acting on $\mathcal{T}_p$.

\begin{prop}\label{group action on cyclic quiver}
Let $G$ be a finite abelian group acting on a tube $(\mathcal{T}_p,\tau)$, i.e., there exists a group homomorphism $\varphi: G\rightarrow \Aut(\mathcal{T}_p)$. Assume $|G|=n$ and $\text{Im}(\varphi)=\langle \tau^{k}\rangle$ for some minimal $1\leq k\leq n$. Then we have an equivalence \begin{equation}\label{tube}\mathcal{T}_p^{\,G}\cong\underbrace{\mathcal{T}_{k}\times\cdots\times\mathcal{T}_{k}}_{\frac{nk}{p} \ \text{times}}.
\end{equation}
\end{prop}

\begin{proof}
Let $Q=(Q_0,Q_1)$ be a cyclic quiver with vertex set $Q_0=\overline{\mathbb{Z}}_{p}:=\{1,\;2,\;\cdots,\;p\}$ and arrow set $Q_1=\{\alpha_i\colon i\rightarrow i-1\,|\,i\in Q_0\}.$ Denote by $\text{rep}^{\text{nil}}_\mathbf{k}Q$ the category of nilpotent representations of $Q$ over $\mathbf{k}$. Then there is an equivalence \begin{equation}\label{nilpotent category of Cp}\text{rep}^{\text{nil}}_\mathbf{k}Q\cong \mathcal{T}_p.
 \end{equation} Hence, the group homomorphism $\varphi: G\rightarrow \Aut(\mathcal{T}_p)$ induces a group action of $G$ on $\text{rep}^{\text{nil}}_\mathbf{k}Q$, which satisfies $g(\alpha_{i})=\alpha_{g(i)}$ for any $g\in G$ and $i\in Q_0$.

Obviously, there is an induced $G$-action on the set $Q_0$. Denote by $O_i$ the $G$-orbit of $i\in Q_0$, and {{let $I$ be a set of  representatives of the distinct orbit classes.}} Then we have $|O_i|=\frac{p}{k}$ and $Q_0=\bigcup_{i\in I}O_i$.

Denote by $G_i$ the stablizer group of $i\in Q_0$. Then we have $|G_i|=\frac{nk}{p}$. Since $G$ is abelian, $G_i$ has $\frac{nk}{p}$-many irreducible representations $\rho_{iu}$'s for $0\leq u\leq \frac{nk}{p}-1$ (each $\rho_{iu}$ has {{dimension}} $1$), where we assume $\rho_{i0}$ is the trivial representation.

We construct a quiver $Q_G=(Q_{G,0},Q_{G,1})$ as follows: the set of vertices
$$Q_{G,0}:=\bigcup_{i\in I}{\{i\}}\times \text{irr}({G_i}),$$
the number of arrows from the vertices $(i,\rho_{iu})$ to $(j,\rho_{jv})$ equals to
$$|\text{Arr}((i,\rho_{iu}),(j,\rho_{jv}))|:=
  \left\{
  \begin{array}{ll}
    1& {\rm \;if\;} u=v {\rm \;and\;} j=i-1;\\
    0&  {\rm \;if\;else.}
  \end{array}
\right.$$
It is easy to see that $Q_G$ is $\frac{nk}{p}$-copies of cyclic quivers with $k$ vertices. Hence \begin{equation}\label{nilpotent category of Q_G}\text{rep}^{\text{nil}}_\mathbf{k}Q_G\cong \mathcal{T}_{k}\times\cdots\times\mathcal{T}_{k} \quad (\frac{nk}{p} \;\;\text{times}).
\end{equation}

Let $M_{ij}$ be the $\mathbf{k}$-vector space spanned by the arrows in $Q$ from $i$ to $j$, that is, $M_{ij}={\bf{k}}\alpha_i$ for $j=i-1$ and zero otherwise. We may regard $M_{ij}$ as a $\mathbf{k}[G_i\cap G_j]$-module by restricting the action of $G$. Recall that $g(\alpha_{i})=\alpha_{g(i)}$ for any $g\in G$ and $i\in Q_0$. We have $M_{ij}\cong \rho_{i0}$ for $j=i-1$ and zero otherwise. Hence,
$$\dim_{\bf{k}}\text{Hom}_{\mathbf{k}[G_i\cap G_j]}(\rho_{iu},\rho_{jv}\otimes_\mathbf{k} M_{ij})= |\text{Arr}((i,\rho_{iu}),(j,\rho_{jv}))|.$$
Denote by $(\mathbf{k}Q)G$ the skew group algebra of ${\bf{k}}Q$ by $G$. Combining with \cite[Theorem 1]{Demonet2010} and \cite[Example 2.6]{Chenxiaowu2017}, we obtain
$$(\mod \mathbf{k}Q)^G\stackrel{\sim}\longrightarrow\mod ((\mathbf{k}Q)G)\stackrel{\sim}\longrightarrow\mod \mathbf{k}Q_G.$$
It follows that
$$ (\text{rep}_\mathbf{k}Q)^G\cong \text{rep}_\mathbf{k}Q_G.$$
By considering their nilpotent counterparts, combing with \eqref{nilpotent category of Cp} and \eqref{nilpotent category of Q_G}, we obtain \eqref{tube}.
\end{proof}

\section{Weighted projective curves}

A \emph{weighted projective curve} or WPC for short, is a pair $\mathbf{X}=(\mathbb{X},\omega)$ consisting of a smooth projective curve $\mathbb{X}$ and a function $\omega$: $\mathbb{X}\rightarrow \mathbb{N}_{\geq 1}$ such that $\omega(x)\textgreater 1$ for only finitely many points.
The points $\lambda_1,\lambda_2,\cdots,\lambda_t$ with $\omega(\lambda_i)\textgreater 1$ are called \emph{weighted points}, the other ones are called \emph{ordinary points}, and $p_i=\omega(\lambda_i)$ ($1\leq i\leq t$) are called the \emph{weights} of $\mathbf{X}$. We often write $\mathbf{X}=\mathbb{X}(\bf{p};\boldsymbol\lambda)$, where ${\bf p}=(p_1, p_2, \cdots, p_t)$ and $\boldsymbol{\lambda}=(\lambda_1, \lambda_2, \cdots, \lambda_t)$.

The \emph{orbifold Euler characteristic} of $\mathbf{X}$ is
$$\chi_{\mathbf{X}}:=\chi_{\mathbb{X}}-\sum_{i=1}^t(1-\frac{1}{p_i}),$$
where $\chi_{\bbX}=2(1-g_{\bbX})$ and $g_\bbX=\dim_{\bf{k}} \textup{Ext}^1(\mathcal{O}_\bbX,\mathcal{O}_\bbX)$.

The category of coherent sheaves ${\rm coh}\mbox{-}\mathbf{X}$ is obtained from ${\rm coh}\mbox{-}\mathbb{X}$ (the category of algebraic coherent sheaves
on $\bbX$) by iteratively applying for each weighted point $\lambda_i$ the so-called \emph{$p_i$-cycle constructio}n; see \cite{Lenzing1998}.

\subsection{Axiomatic description of ${\rm coh}\mbox{-}\mathbf{X}$ for WPCs}

The category ${\rm coh}\mbox{-}\mathbf{X}$ has an axiomatic description as below; see \cite{Lenzing2019,LenzingReiten2006}.

\begin{itemize}
\item[$(\mathbf {H\;1})$] $\mathcal{A}$ is a connected abelian $\mathbf k$-linear noetherian category;

\item[$(\mathbf {H\;2})$] $\mathcal{A}$ is Ext-finite, that is, all morphism and extensions spaces in $\mathcal{A}$ are finite dimensional $\mathbf k$-vector spaces;

\item[$(\mathbf {H\;3})$] $\mathcal{A}$ satisfies Serre duality $D\text{Ext}_{\mathcal{A}}^{1}(X,Y)\cong \text{Hom}_{\mathcal{A}}(Y,\tau X),$ where $\tau$ is a self-equivalence of $\mathcal{A}$;

\item[$(\mathbf {H\;4})$] $\mathcal{A}$ contains an object of infinite length;

\item[$(\mathbf {H\;5})$] each object in the quotient category $\mathcal{A}/\mathcal{A}_0$ has finite length;

\item[$(\mathbf {H\;6})$] $\mathcal{A}$ has infinitely many points.
\end{itemize}

Here, we recall from \cite[Proposition 1.1]{LenzingReiten2006} that under the assumption $(\textup {H\;1})-(\textup {H\;5})$, the length subcategory $\mathcal{A}_0$ of $\mathcal{A}$ has the form $\coprod_{x\in C}\mathcal{U}_{x}$, where the members of $C$ are called \emph{points} of $\mathcal{A}$, and each $\mathcal{U}_{x}$ is a connected uniserial length category of finite $\tau$-period. More precisely, each $\mathcal{U}_{x}$ is a tube $\mathcal{T}_{\omega(x)}$.

\begin{prop}\label{wpc=h1-h6}\cite{Lenzing2019}
A small category is equivalent to a category of coherent sheaves on a weighted projective curve if and only if it satisfies $(\textup {H\;1})-(\textup {H\;6})$.
\end{prop}

\subsection{Group action preserves axioms for WPC}

In this subsection, we show that the axioms $(\textup {H\;1})-(\textup {H\;6})$ for WPC's are preserved under certain group actions. Therefore, the equivariantization with respect to certain group actions preserves the categories of coherent sheaves over weighted projective curves.

\begin{thm}\label{main result}
Let $\mathbf{X}$ be a weighted projective curve and $G$ be a finite subgroup of $\Aut({\rm coh}\mbox{-}\mathbf{X})$. Then
$$({\rm coh}\mbox{-}\mathbf{X})^{G}\stackrel{\sim}\longrightarrow {\rm coh}\mbox{-}\mathbf{Y}$$ for some weighted projective curve $\mathbf{Y}$.
\end{thm}

\begin{proof}
We denote by $\mathcal{A}:={\rm coh}\mbox{-}\mathbf{X}$ throughout the proof. By Proposition \ref{wpc=h1-h6}, we know that $\mathcal{A}$ satisfies $(\textup {H\;1})-(\textup {H\;6})$, and remain to prove so does $\mathcal{A}^{G}$.

In order to prove $\mathcal{A}^{G}$ satisfies (\textup {H\;1}), the only non-trivial thing is to show that each object $(X,\alpha)$ in $\mathcal{A}^{G}$ is noetherian. For this we let $$(X^{1},\alpha^{1})\subseteq(X^{2},\alpha^{2})\subseteq\cdots\subseteq(X^{n},\alpha^{n})\subseteq\cdots$$ be an ascending chain of subobjects of $(X,\alpha)$ in $\mathcal{A}^{G}$. Since the forgetful functor $U\colon \mathcal{A}^{G}\rightarrow \mathcal{A}$ is exact, we obtain an ascending chain $X^{1}\subseteq X^{2}\subseteq\cdots\subseteq X^{n}\subseteq\cdots$ of subobjects of $X$ in $\mathcal{A}$. By assumption $X$ is noetherian in $\mathcal{A}$, hence there exists some integer $m$, such that $X^{i}= X^{i+1}$ for any $i\geq m$. By the exactness of $U$ again we obtain $(X^{i},\alpha^{i})=(X^{i+1},\alpha^{i+1})$ for any $i\geq m$.
That is, $(X,\alpha)$ is noetherian.

Recall that the forgetful functor $U\colon \mathcal{A}^{G}\rightarrow \mathcal{A}$ is faithful. We have
$$\text{Hom}_{\mathcal{A}^G}((X,\alpha),(Y,\beta))\subseteq\text{Hom}_{\mathcal{A}}(U(X,\alpha),U(Y,\beta))
=\text{Hom}_{\mathcal{A}}(X,Y).$$
Then $\mathcal{A}$ is Hom-finite implies so is $\mathcal{A}^G$. Moreover, it follows from \cite[Theorem 3.6]{Chenxiaowu2014} that $\mathcal{A}^{G}$ has Serre duality of the following form for any $(X,\alpha),(Y,\beta)\in \mathcal{A}^{G}$: $$D\text{Ext}_{\mathcal{A}^G}^{1}((X,\alpha),(Y,\beta))\cong \text{Hom}_{\mathcal{A}^G}((Y,\beta),\tau^{G}(X,\alpha)),$$
where $\tau^{G}$ is an autoequivalence of $\mathcal{A}^G$. Then by \cite[Proposition 1.8.1]{ChenxiaowuKrause2009}, $\mathcal{A}^{G}$ is hereditary, and then the Ext-finiteness follows. Thus $\mathcal{A}^{G}$ satisfies $(\textup {H\;2})$ and $(\textup {H\;3})$.

By assumption $\mathcal{A}$ contains a strict descending chain $$X_1\supset X_2\supset\cdots\supset X_n\cdots.$$ Since the induction functor $F\colon \mathcal{A}\rightarrow \mathcal{A}^{G}$ is exact, we obtain a descending chain
\begin{equation}\label{descending chain under F}F(X_1)\supseteq F(X_2)\supseteq\cdots\supseteq F(X_n)\cdots.\end{equation}
We claim that the chain \eqref{descending chain under F} is not stationary, which proves $(\textup {H\;4})$.
For contradiction, we assume there exists some $m\geq 1$, such that $F(X_i)=F(X_{i+1})$ for any $i\geq m$. Then $UF(X_i)= UF(X_{i+1})$, which follows that $X_i=X_{i+1}$, a contradiction.

We have shown that $\mathcal{A}^{G}$ satisfies $(\textup {H\;1})-(\textup {H\;4})$. By using \cite[Proposition 4.9]{LenzingReiten2006}, we know that $\mathcal{A}^{G}$ satisfies $(\textup{H\;5})$ automatically.

Finally, we prove that $\mathcal{A}^{G}$ satisfies $(\textup{H\;6})$. By assumption $G$ is a finite subgroup of $\Aut(\mathcal{A})$, hence each $g\in G$ defines an automorphism of $\mathcal{A}$, which restricts to an automorphism of $\mathcal{A}_{0}$. Recall that { $\mathcal{A}_{0}=\coprod_{x\in C}\mathcal{U}_{x}$}. Hence each $g\in G$ induces a permutation of $C$. Since $G$ is finite, it follows that $\mathcal{A}^{G}$ has infinitely many points. This finishes the proof.
\end{proof}

\section{Weighted projective lines and proof of the main results}

This section is devoted to proving Theorem \ref{trichotomy} and Proposition \ref{change of weight}. For these we need to recall the category of coherent sheaves over weighted projective lines in more details, including the structure of its automorphism group.

\subsection{The string group $\bbL(\mathbf{p})$}
{Let $t\geq 0$ be an integer. Let ${\bf p}=(p_1, p_2, \cdots, p_t)$ be a sequence
of integers with each $p_i\geq 2$.
The \emph{string group} $\bbL({\bf p})$ of type ${\bf p}$ is an abelian group (written additively) on generators $\vec{x}_1, \vec{x}_2, \cdots, \vec{x}_t$, subject to the relations $p_1\vec{x}_1=p_2\vec{x}_2=\cdots = p_t\vec{x}_t$, where this common element is denoted by $\vec{c}$ and called the \emph{canonical element} of $\bbL(\mathbf{p})$.
The string group $\bbL(\mathbf{p})$ has rank one, where $\vec{c}$ has infinite order. There is an isomorphism of abelian groups
\begin{align}\label{isomorphism of L} \bbL(\mathbf{p})/\mathbb{Z}\vec{c}\stackrel{\sim}\longrightarrow \prod_{i=1}^t \mathbb{Z}/p_i\mathbb{Z},\end{align}
sending $\vec{x}_i+\mathbb{Z}\vec{c}$ to the vector $(0, \cdots,0,  \bar{1}, 0, \cdots, 0)$ with $\bar{1}$ on the $i$-th component. Using (\ref{isomorphism of L}), each element $\vec{x}$ in $\bbL(\mathbf{p})$ can be uniquely written in its \emph{normal form}
\begin{align}\label{equ:nor}
\vec{x}=\sum_{i=1}^t l_i\vec{x}_i+l\vec{c},
\end{align}
where $0\leq l_i\leq p_i-1$ for $1\leq i\leq t$ and $l\in \mathbb{Z}$.

The \emph{torsion subgroup} $t\bbL(\mathbf{p})$ of $\bbL(\mathbf{p})$ consists of all the torsion elements $\vec{x}$ in $\bbL(\mathbf{p})$, i.e., $n\vec{x}=0$ for some nonzero integer $n$. Denote by $p={\rm l.c.m.}(p_1,p_2,\cdots,p_t)$. There is a surjective group homomorphism $\delta\colon \bbL(\mathbf{p})\rightarrow \mathbb{Z}$ given by $\delta(\vec{x}_i)=\frac{p}{p_i}$ for $1\leq i\leq t$.
The following description of the torsion group $t\bbL(\mathbf{p})$ plays an important role in this paper.

\begin{lem}\label{torsiongp}
Let ${\bf{p}}=(p_1,\cdots,p_t)$. Denote by $p=\textup{l.c.m.}(p_1,\cdots,p_t)$ and $d_{ij}=\textup{g.c.d.}(p_i, p_j)$ for any $1\leq i\neq j\leq t$. Then the following hold:

\begin{itemize}
\item[$(1)$] $\vec{x}\in t\bbL({\bf{p}})$ if and only if $p\vec{x}=0$;
\item[$(2)$] there is an exact sequence $0\rightarrow t\mathbb{L}(\mathbf{p})\rightarrow \mathbb{L}(\mathbf{p})\xrightarrow{\delta}\mathbb{Z}\rightarrow 0$;
\item[$(3)$] $|t\bbL({\bfp})|=\frac{p_1p_2\cdots p_t}{p}$;
\item[$(4)$] $t\bbL({\bf{p}})=\langle\frac{p_i}{d_{ij}}\vec{x}_i-\frac{p_j}{d_{ij}}\vec{x}_j|1\leq i<j\leq t\rangle$.
\end{itemize}
\end{lem}

\begin{proof}
(1) and (2) are taken from \cite[1.2]{GeigleLenzing1987}. Consequently, there is an induced exact sequence
$$0\rightarrow t\mathbb{L}(\mathbf{p})\rightarrow \mathbb{L}(\mathbf{p})/\mathbb{Z}\vec{c} \xrightarrow{\delta}\mathbb{Z}/p\mathbb{Z}\rightarrow 0;$$
hence (3) follows from \eqref{isomorphism of L} immediately; compare \cite{Chenxiaowu2015}.

In the following we prove (4). Denote by $H=\langle\frac{p_i}{d_{ij}}\vec{x}_i-\frac{p_j}{d_{ij}}\vec{x}_j|1\leq i<j\leq t\rangle$. It is obvious that $H\subseteq t\bbL(\bfp).$ We will show $t\bbL(\bfp)\subseteq H$ by induction on $t$. If $t\leq 2$, it is trivial.
We assume that $t>2$ from now on. Let ${\bf{p}'}=(p_1,\cdots,p_{t-1})$. We denote by $p'=\textup{l.c.m.}(p_1,\cdots,p_{t-1})$ and $d=\textup{g.c.d.}(p_t,p')$. Then clearly we have $p=\frac{p'p_t}{d}$.

For any $\vec{x}\in t\bbL({\bfp})$, we can write $\vec{x}=a\vec{x}_t+\vec{y}$ with $\vec{y}\in \bbL({\bf{p}'})$. If $a=0$, then $\vec{x}=\vec{y}\in t\bbL({\bfp}')\subseteq H$ by induction, we are done. We assume $1\leq a\leq p_t-1$ in the following. By (1), $$0=p\vec{x}=\frac{p'p_t}{d}\vec{x}=\frac{ap'p_t}{d}\vec{x}_t+\frac{p'p_t}{d}\vec{y}
=\frac{ap'}{d}\vec{c}+\frac{p_t}{d}(p'\vec{y}).$$ Since $p'\vec{y}\in \bbZ\vec{c}$, we have $\frac{p_t}{d}|\frac{ap'}{d}$. Note that $\textup{g.c.d.}(\frac{p_t}{d},\frac{p'}{d})=1$, hence $\frac{p_t}{d}|a$, i.e., $p_t|da$.

Write $d=\Pi_{j\in J}q_j^{l_j}$, where $q_j$'s are pairwise distinct primes. For any $1\leq i\leq t-1$, let $J_i=\{j\in J\,\big|\;q_j^{l_j}\mid p_i, \textup{\ and\ } q_j^{l_j}\nmid p_k \textup{\ for\ any\ } k<i\}$ and take $d_i=\Pi_{j\in J_i}q_j^{l_j}$. Then $d_i|p_i$ and $\textup{g.c.d.}(d_i,d_j)=1$ for $1\leq i\neq j\leq t-1$, and $d=d_1d_2\cdots d_{t-1}$.

Since $\textup{g.c.d.}(d_i,d_j)=1$, we have $\textup{g.c.d.}(\frac{d}{d_1},\frac{d}{d_2},\cdots,\frac{d}{d_{t-1}})=1$. It follows that there exist $u_i$ ($1\leq i\leq t-1$) such that $\sum_{i=1}^{t-1}u_i\frac{d}{d_i}=1$. Let $\vec{y}_0=\sum_{i=1}^{t-1}u_i\frac{p_i}{d_i}\vec{x}_i$, then
\begin{align*}
\vec{y}_0-\frac{p_t}{d}\vec{x}_t&=\sum_{i=1}^{t-1}u_i\frac{p_i}{d_i}\vec{x}_i-\sum_{i=1}^{t-1}\frac{u_id}{d_i}\cdot\frac{p_t}{d}\vec{x}_t
=\sum_{i=1}^{t-1}\frac{u_id_{it}}{d_i}(\frac{p_i}{d_{it}}\vec{x}_i-\frac{p_t}{d_{it}}\vec{x}_t)\in H.
\end{align*}

Observe that $\vec{x}=-\frac{da}{p_t}(\vec{y}_0-\frac{p_t}{d}\vec{x}_t)+\vec{z}$, where $\vec{z}=\vec{y}+\frac{da}{p_t}\vec{y}_0$ (recall that $p_t|da$). Therefore, $\vec{x}\in t\bbL({\bfp})$ implies $\vec{z}\in t\bbL({\bfp'})$. Then by induction we get $\vec{z}\in H$, which implies $\vec{x}\in H$. This finishes the proof.
\end{proof}

\subsection{Weighted projective lines}
Let ${\boldsymbol\lambda}=(\lambda_1, \lambda_2, \cdots, \lambda_t)$ be a sequence of
pairwise distinct points on the projective line $\mathbb{P}_{\mathbf k}^1$. Such a sequence can be normalized such that $\lambda_1=\infty$, $\lambda_2=0$ and $\lambda_3=1$.
A \emph{weighted projective line} $\mathbb{P}_{\mathbf k}^1(\mathbf{p}; {\boldsymbol\lambda})$ of weight type $\mathbf{p}$ and parameter sequence ${\boldsymbol\lambda}$ is obtained from $\mathbb{P}_{\mathbf k}^1$ by attaching the weight $p_i$ to each point $\lambda_i$ for $1\leq i\leq t$. We will always assume ${\boldsymbol\lambda}$ is normalized in this paper unless stated otherwise,
and simply write $\mathbb{P}_{\mathbf k}^1(\bf p; {\boldsymbol\lambda})=\mathbb{P}_{\mathbf k}^1(\bf p)$ for $t\leq 3$.

The \emph{homogeneous coordinate algebra} $S(\mathbf{p}; {\boldsymbol\lambda})$ of the weighted projective line $\mathbb{P}_{\mathbf k}^1(\mathbf{p};{\boldsymbol\lambda})$ is given by $\mathbf{k}[X_1, X_2, \cdots, X_t]/I$, where the ideal $I$ is generated by $X_i^{p_i}-(X_2^{p_2}-\lambda_iX_1^{p_1})$ for $3\leq i\leq t$. We write $x_i=X_i+I$ in $S(\mathbf{p}; {\boldsymbol\lambda})$.
The algebra $S(\mathbf{p}; {\boldsymbol\lambda})$ is $\bbL(\mathbf{p})$-graded by means of $\deg (x_i)=\vec{x}_i$.

We recall a more convenient description of the category ${\rm coh}\mbox{-}\mathbb{P}_{\mathbf k}^1(\mathbf{p};{\boldsymbol\lambda})$ of coherent sheaves over $\mathbb{P}_{\mathbf k}^1(\mathbf{p};{\boldsymbol\lambda})$ via graded $S(\mathbf{p}; {\boldsymbol\lambda})$-modules. Denote by ${\rm mod}^{\bbL(\mathbf{p})}\mbox{-}S(\mathbf{p}; {\boldsymbol\lambda})$ the abelian category of finitely generated $\bbL(\mathbf{p})$-graded $S(\mathbf{p}; {\boldsymbol\lambda})$-modules, and by ${\rm mod}_0^{\bbL(\mathbf{p})}\mbox{-}S(\mathbf{p}; {\boldsymbol\lambda})$ its Serre subcategory formed by finite dimensional modules. We denote the quotient abelian category by $${\rm qmod}^{\bbL(\mathbf{p})}\mbox{-}S(\mathbf{p}; {\boldsymbol\lambda}):=\frac{{\rm mod}^{\bbL(\mathbf{p})}\mbox{-}S(\mathbf{p}; {\boldsymbol\lambda})}{{{\rm mod}_0^{\bbL(\mathbf{p})}\mbox{-}S(\mathbf{p}; {\boldsymbol\lambda})}}.$$ By \cite[Theorem 1.8]{GeigleLenzing1987} the sheafification functor yields an equivalence
$$
{\rm qmod}^{\bbL(\mathbf{p})}\mbox{-}S(\mathbf{p}; {\boldsymbol\lambda})\stackrel{\sim}\longrightarrow {\rm coh}\mbox{-}\mathbb{P}_{\mathbf k}^1(\mathbf{p};{\boldsymbol\lambda}).
$$
From now on we will identify these two categories.}

{ The \emph{Euler characteristic} of a weighted projective line $\mathbb{P}_{\mathbf k}^1(\bf p; {\boldsymbol\lambda})$ is defined as $$\chi_{\mathbb{P}_{\mathbf k}^1(\bf p; {\boldsymbol\lambda})}=2-\sum_{i=1}^t(1-\frac{1}{p_i})=-\frac{\delta(\vec{\omega})}{p},$$
where $\vec{\omega}=(t-2)\vec{c}-\sum_{i=1}^t \vec{x}_i$ is the \emph{dualizing element} of $\bbL(\mathbf{p})$. The weighted projective line $\mathbb{P}_{\mathbf k}^1(\bf p; {\boldsymbol\lambda})$ is called of \emph{domestic, tubular} or \emph{wild} type provided that $\chi_{\mathbb{P}_{\mathbf k}^1(\bf p; {\boldsymbol\lambda})}>0$, $\chi_{\mathbb{P}_{\mathbf k}^1(\bf p; {\boldsymbol\lambda})}=0$ or $\chi_{\mathbb{P}_{\mathbf k}^1(\bf p; {\boldsymbol\lambda})}<0$, respectively. } More precisely, we have the following trichotomy according to the weight type $\mathbf{p}$ (up to permutation of weights):
\begin{itemize}
  \item [(i)] domestic type: $(), (p_1), (p_{1},p_{2})$, $(2,2,p_3)$, $(2,3,3)$, $(2,3,4)$ and $(2,3,5)$;
  \item [(ii)] tubular type: $(2,2,2,2)$, $(3,3,3)$, $(4,4,2)$  and $(6,3,2)$;
  \item [(iii)] wild type: all the other cases.
\end{itemize}

\subsection{Automorphism group for weighted projective lines}

The automorphism group $\Aut({\rm coh}\mbox{-}\mathbb{P}_{\mathbf k}^1(\mathbf{p};{\boldsymbol\lambda}))$ of the category ${\rm coh}\mbox{-}\mathbb{P}_{\mathbf k}^1(\mathbf{p};{\boldsymbol\lambda})$ consists of isomorphism classes of $\mathbf k$-linear self-equivalences on ${\rm coh}\mbox{-}\mathbb{P}_{\mathbf k}^1(\mathbf{p}; {\boldsymbol\lambda})$.
Let $\Aut(\mathbb{P}_{\mathbf k}^1(\mathbf{p};{\boldsymbol\lambda}))$ be the subgroup of $\Aut({\rm coh}\mbox{-}\mathbb{P}_{\mathbf k}^1(\mathbf{p};{\boldsymbol\lambda}))$ of automorphisms fixing the structure sheaf.
By \cite[Proposition 3.1]{LenzingMelter2000}, $\Aut(\mathbb{P}_{\mathbf k}^1(\mathbf{p};{\boldsymbol\lambda}))$ is isomorphic to the subgroup of $\textup{PSL}(2,\mathbf k)$ of automorphisms of $\mathbb{P}_{\mathbf k}^1$ preserving weights. If $\mathbb{P}_{\mathbf k}^1(\mathbf{p};{\boldsymbol\lambda})$ has $t\geq 3$ exceptional points,
$\textup{Aut}(\mathbb{P}_{\mathbf k}^1(\mathbf{p};{\boldsymbol\lambda}))$ is finite since it is isomorphic to a subgroup of a direct product of some symmetric group by \cite[Corollary 3.2]{LenzingMelter2000}. In particular,
 if the weights $p_1,\cdots,p_t$ are pairwise distinct, then $\textup{Aut}(\mathbb{P}_{\mathbf k}^1(\mathbf{p};{\boldsymbol\lambda}))$ is trivial.

 We have the following description of the automorphism group $\Aut({\rm coh}\mbox{-}\mathbb{P}_{\mathbf k}^1(\mathbf{p};{\boldsymbol\lambda}))$ for a weighted projective line.

\begin{prop}\cite[Theorem 3.4]{LenzingMelter2000}
{There is a split-exact sequence
\begin{align}\label{automorphism group}1\rightarrow \mathbb{L}(\mathbf{p})\rightarrow{\Aut}({\rm coh}\mbox{-}\mathbb{P}_{\mathbf k}^1(\mathbf{p}; {\boldsymbol\lambda}))\rightarrow{\Aut}(\mathbb{P}_{\mathbf k}^1(\mathbf{p}; {\boldsymbol\lambda}))\rightarrow 1.
\end{align}
In particular, $$\Aut({\rm coh}\mbox{-}\mathbb{P}_{\mathbf k}^1(\mathbf{p};{\boldsymbol\lambda}))= \mathbb{L}(\mathbf{p})\rtimes {\Aut}(\mathbb{P}_{\mathbf k}^1(\mathbf{p};{\boldsymbol\lambda})).$$}
\end{prop}}

The automorphisms from $\Aut(\mathbb{P}_{\mathbf k}^1(\mathbf{p};{\boldsymbol\lambda}))$ or $\bbL({\bf p})$ are  called \emph{geometric automorphisms} or \emph{degree-shift automorphisms} respectively in this paper.

\subsection{Proof of Theorem \ref{trichotomy}}

According to Theorem \ref{main result}, we know that $\mathbf{Y}$ is a weighted projective curve. In the following we assume $\mathbf{Y}=\mathbb{Y}(\mathbf{q};{\boldsymbol\mu})$ with $\mathbf{q}=(q_1,q_2,\cdots, q_{t'})$.

Let $H^{\ast}$ be the {{character}} group of $H$. Since $H$ is commutative, by \cite[Theorem 4.6]{CCR} there exists an $H^{\ast}$-action on $\coh\mathbf{Y}$ with an equivalence $$({\rm coh}\mbox{-}\mathbf{Y})^{H^{\ast}}\stackrel{\sim}\longrightarrow {\rm coh}\mbox{-}\mathbb{P}_{\mathbf k}^1(\mathbf{p};{\boldsymbol\lambda}).$$ Moreover, by the definition of $H^{\ast}$-action on $\coh\mathbf{Y}$ (see \cite[Section 4.3]{CCR}), we know that $H^{\ast}$ fixes the structure sheaf of $\mathbf{Y}$, hence $H^{\ast}$ acts on the curve $\mathbf{Y}$ with a quotient stack $\mathbf{Y}/H^{\ast}\cong \mathbb{P}_{\mathbf k}^1(\mathbf{p};{\boldsymbol\lambda})$. Then by Theorem 2 and (2.3) in \cite{Lenzing2017},
\begin{align}\label{Hurwitz formula application2}|H^{\ast}|\cdot\chi_{\mathbb{P}_{\mathbf k}^1(\mathbf{p}; {\boldsymbol\lambda})}=\chi_{\mathbf{Y}}=\chi_{\bbY}-\sum\limits_{i=1}^{t'}(1-\frac{1}{q_i})= 2(1-g_{\bbY})-\sum\limits_{i=1}^{t'}(1-\frac{1}{q_i}).
\end{align}

Therefore, $\mathbb{P}_{\mathbf k}^1(\mathbf{p};{\boldsymbol\lambda})$ is of domestic type, i.e., $\chi_{\mathbb{P}_{\mathbf k}^1(\mathbf{p};{\boldsymbol\lambda})}>0$, if and only if $g_{\mathbb{Y}}=0$ and $\sum\limits_{i=1}^{t'}(1-\frac{1}{q_i})<2$, or equivalently, $\mathbf{Y}$ is a weighted projective line of domestic type. $\mathbb{P}_{\mathbf k}^1(\mathbf{p};{\boldsymbol\lambda})$ is of tubular type, i.e., $\chi_{\mathbb{P}_{\mathbf k}^1(\mathbf{p};{\boldsymbol\lambda})}=0$, if and only if one of the following holds: $$(i)\, g_{\mathbb{Y}}=0\;\text{and}\; \sum\limits_{i=1}^{t'}(1-\frac{1}{q_i})=2; \quad (ii)\, g_{\mathbb{Y}}=1\;\text{and}\; \sum\limits_{i=1}^{t'}(1-\frac{1}{q_i})=0.$$
Hence $\mathbf{Y}$ is a weighted projective line of tubular type or a smooth elliptic curve respectively.
$\mathbb{P}_{\mathbf k}^1(\mathbf{p}; {\boldsymbol\lambda})$ is of wild type if and only if $\chi_{\mathbb{P}_{\mathbf k}^1(\mathbf{p}; {\boldsymbol\lambda})}<0$, if and only if one of the following holds: $$(i)\, g_{\mathbb{Y}}=0\;\text{and}\; \sum\limits_{i=1}^{t'}(1-\frac{1}{q_i})>2; \quad (ii)\, g_{\mathbb{Y}}=1\;\text{and}\; \sum\limits_{i=1}^{t'}(1-\frac{1}{q_i})>0; \quad (iii)\, g_{\mathbb{Y}}\geq 2.$$ That is, $\mathbf{Y}$ is a weighted projective line of wild type, or a (non-trivial) weighted elliptic curve, or a weighted projective curve with $g_{\mathbb{Y}}\geq 2$ respectively.
This proves the trichotomy result of Theorem \ref{trichotomy}.

Now we will classify all the curves $\mathbf{Y}=\mathbb{Y}(\mathbf{q};{\boldsymbol\mu})$ when $\mathbb{P}_{\mathbf k}^1(\mathbf{p}; {\boldsymbol\lambda})$ is of domestic type or of tubular type. By explicit description on the torsion group $t\bbL(\bfp)$ in Lemma \ref{torsiongp}, we can obtain all the possibilities of $H\leq t\bbL(\bfp)$ as listed in Tables \ref{table 1} and \ref{table 2}.
By using Proposition \ref{change of weight} we can obtain the data $g_{\bbY}$ and $\bfq$ for each $H$.
This already finishes the classification of $\mathbf{Y}$'s when $\mathbb{P}_{\mathbf k}^1(\mathbf{p}; {\boldsymbol\lambda})$ is of domestic type, since domestic weighted projective lines are determined by their weight data $\mathbf{p}$ (by normalization assumption on parameters ${\boldsymbol\lambda}$); compare with \cite[Table 1]{CLRZ}.

If $\mathbb{P}_{\mathbf k}^1(\mathbf{p}; {\boldsymbol\lambda})$ is of tubular type, we still need to deal with the parameter ${\boldsymbol\mu}$ for $\mathbf{Y}=\mathbb{Y}(\mathbf{q};{\boldsymbol\mu})$. In this case, it is easy to check case by case that any finite subgroup $H$ of $\bbL(\bf p)$ satisfies
$H\supseteq\bbZ\vec{\omega}$ or $H$ is of Cyclic type or Klein type in the sense of \cite{CLRZ}.
According to the trichotomy result above and \cite[Proposition 3.8]{CLRZ}, we know that $\mathbf{Y}=\mathbb{Y}(\mathbf{q};{\boldsymbol\mu})$ is an smooth elliptic curve for the first case, and $\mathbf{Y}$ is a weighted projective line of tubular type for the remaining two cases, where the parameter data has been determined by \cite[Table 3]{CLRZ}.

For $H=\bbZ\vec{\omega}$, the parameter of $\mathbf{Y}=\mathbb{Y}(\mathbf{q};{\boldsymbol\mu})$ has been calculated by \cite[Section 7]{CCZ}. To finish the proof we remain to determine ${\boldsymbol\mu}$ when $({\bf{p}},H)$ are given by the following:
$$\begin{array}{ll}
(\textup{i})\, \big((2,2,2,2),\langle\vec{x}_1-\vec{x}_2,\vec{x}_3-\vec{x}_4\rangle\big),&(\textup{ii})\, \big((2,2,2,2),t\bbL(2,2,2,2)\big),\\
(\textup{iii})\, \big((4,4,2),t\bbL(4,4,2)\big),&(\textup{iv})\, \big((3,3,3),t\bbL(3,3,3)\big).
\end{array}$$
In the following we only consider the case (iv), since the other cases can be treated similarly.

Let ${\bf{p}}=(3,3,3)$ and $S({\bf{p}})={\bf{k}}[X_1,X_2,X_3]/(f)$, where $f=X_1^3-X_2^3+X_3^3$. Let $x_i=X_i+(f)$ in $S({\bf{p}})$ for $i=1,2,3$. Recall that $S({\bf{p}})$ is an $\bbL({\bf{p}})$-graded algebra with $\textup{deg}(x_i)=\vec{x}_i$ for $i=1,2,3$. Note that $\bbL({\bf{p}})=\bbZ\vec{x}_1\oplus t\bbL({\bf{p}})$. Hence $\bbL({\bf{p}})/t\bbL({\bf{p}})\cong \mathbb{Z}$. We define a $\mathbb{Z}$-graded algebra $\overline{S({\bf{p}})}$ as follows: as an ungraded algebra $\overline{S({\bf{p}})}=S({\bf{p}})$, while its homogeneous component is given by $\overline{S({\bf{p}})}_n=\oplus_{\vec{y}\in t\bbL({\bf{p}})}S({\bf{p}})_{n\vec{x}_1+\vec{y}}$ for each $n\in\mathbb{Z}$. By \cite[Proposition 5.2]{CCZ} we have an equivalence of categories
$$\big(\mod^{\bbL({\bf{p}})}S({\bf{p}})\big)^{t\bbL({\bf{p}})}\stackrel{\sim}\longrightarrow\mod^{\mathbb{Z}}\overline{S({\bf{p}})},$$
which restricts to an equivalence $\big(\mod_0^{\bbL({\bf{p}})}S({\bf{p}})\big)^{t\bbL({\bf{p}})}\stackrel{\sim}\longrightarrow\mod_0^{\mathbb{Z}}\overline{S({\bf{p}})}$ of full subcategories consisting of finite dimensional modules. Applying \cite[Corollary 4.4]{CCZ} we have an equivalence of categories\begin{align}\big(\text{qmod}^{\bbL({\bf{p}})}S({\bf{p}})\big)^{t\bbL({\bf{p}})}\stackrel{\sim}\longrightarrow\text{qmod}^{\mathbb{Z}}\overline{S({\bf{p}})}.
\end{align} That is,
\begin{align}({\rm coh}\mbox{-}\mathbb{P}_{\mathbf k}^1(\mathbf{p}; {\boldsymbol\lambda}))^{t\bbL({\bf{p}})}\stackrel{\sim}\longrightarrow {\rm coh}\mbox{-}\bf{E},\end{align}
where $\bf{E}$ is an elliptic curve defined by the Fermat equation $x_1^3-x_2^3+x_3^3=0$.

It is well-known that the elliptic curve $\bf{E}$ has the normal form $y^2-y=x^3-7$ (c.f. \cite[Example 2.5 in Section 1]{Husemoller2004}), hence the $j$-invariant $j(\bf{E})=0$. By assumption, $\omega^2-\omega+1=0$. We know that $$j(\mathbb{E}(\omega))=\frac{2^8(\omega^2-\omega+1)^3}{\omega^2(\omega-1)^2}=0.$$
It follows from \cite[Proposition 1.5 in Section 4]{Husemoller2004} that $\bf{E}\cong \bbE(\mu)$ for any $\mu\in\Gamma(\omega)$. This finishes the proof.

\subsection{Proof of Proposition \ref{change of weight}}

Let ${\rm coh}_0\mbox{-}\mathbb{P}_{\mathbf k}^1(\mathbf{p};{\boldsymbol\lambda})=\coprod_{x\in \mathbb{P}_{\mathbf k}^1}\mathcal{U}_{x}$, where each $\mathcal{U}_{\lambda_i}$ for $1\leq i\leq t$ is a stable tube $\mathcal{T}_{p_i}$ of rank $p_i$, and other $\mathcal{U}_{x}$'s are stable tubes of rank $1$. Note that the degree-shift action of $H$ on ${\rm coh}\mbox{-}\mathbb{P}_{\mathbf k}^1(\mathbf{p};{\boldsymbol\lambda})$ restricts to an $H$-action on $\mathcal{U}_{x}$ for each $x\in \mathbb{P}_{\mathbf k}^1$. That is, there exists a group homomorphism $\varphi_x: H\to \Aut(\mathcal{U}_{x})$ for any $x$.
For each exceptional point $\lambda_j\in \mathbb{P}_{\mathbf k}^1$, since ${\rm g.c.d.}(a_{1j}, a_{2j}, \cdots, a_{sj}, p_j)=d_j$, we know that $d_j$ is the minimal positive integer such that $\text{Im}(\varphi_{\lambda_j})=\langle \tau^{d_j}\rangle$. Then by Proposition \ref{group action on cyclic quiver} we have an equivalence $$(\mathcal{U}_{\lambda_j})^H\cong \underbrace{\mathcal{T}_{d_j}\times\cdots\times\mathcal{T}_{d_j}}_{n_j \ \text{times}}.$$
Similarly, for any $x\notin\{\lambda_1, \lambda_2, \cdots, \lambda_t\}$, $(\mathcal{U}_x)^H$ is equivalent to a product of stable tubes of rank $1$. Hence $${\bf q}=(\underbrace{d_1,\cdots, d_1}_{n_1 \ \text{times}}, \underbrace{d_2,\cdots, d_2}_{n_2 \ \text{times}},\cdots, \underbrace{d_t,\cdots, d_t}_{n_t \ \text{times}}).$$

According to \cite[Theorem 2]{Lenzing2017}, we obtain $$\chi_{\mathbb{P}_{\mathbf k}^1(\mathbf{p}; {\boldsymbol\lambda})}=\frac{\chi_{\mathbf{Y}}}{|H^{\ast}|}.$$
Recall that $\chi_{\mathbb{P}_{\mathbf k}^1(\mathbf{p}; {\boldsymbol\lambda})}=2-\sum_{j=1}^t(1-\frac{1}{p_j})$. Moreover, by \cite[(2.3)]{Lenzing2017}, we have
$$\chi_{\mathbf{Y}}=2(1-g_{\bbY})-\sum_{j=1}^t(1-\frac{1}{d_j})n_j=2(1-g_{\bbY})-n\sum_{j=1}^t\frac{d_j-1}{p_j}.$$
It follows that
$$g_{\bbY}=\frac{n}{2}\sum_{j=1}^t(1- \frac{d_j}{p_j})-n+1.$$
This finishes the proof of Proposition \ref{change of weight}.

\section{Applications}

The equivalence \eqref{main equivalence} provides a new approach to investigate the category of coherent sheaves of weighted projective curves via that of weighted projective lines. In this section, we give some applications of \eqref{main equivalence} to hyperelliptic curves of arbitrary genus and smooth projective curves of genus two. Moreover, we also establish links to Arnold's exceptional unimodal singularities.

\subsection{Weighted projective lines and smooth projective curves}
The following result is an immediate consequence of Proposition \ref{change of weight}, which gives a necessary and sufficient condition for $\mathbf{Y}$ being a smooth projective curve.

{\begin{cor}\label{smooth1}
Keep the notations and assumptions in Proposition \ref{change of weight}. Then the following are equivalent:
\begin{itemize}
\item[$(1)$]  $\mathbf{Y}$ is a smooth projective curve;

\item[$(2)$] ${\rm g.c.d.}(a_{1j}, a_{2j}, \cdots, a_{sj}, p_j)=1$ for each $1\leq j\leq t$.
\end{itemize}
\end{cor}}

\begin{proof}
By Proposition \ref{change of weight}, $\mathbf{Y}$ is a smooth projective curve if and only if $d_j=1$ for $1\leq j\leq t$, if and only if ${\rm g.c.d.}(a_{1j}, a_{2j}, \cdots, a_{sj}, p_j)=1$ for each $1\leq j\leq t$.
\end{proof}

Let $q$ be a prime number. For $a,r\in\mathbb{N}$, we introduce the notation $q^a\parallel r$ by means of $q^a|r$ but $q^{a+1}\nmid r$. For ${\bf{p}}=(p_1, p_2, \cdots, p_t)$, let $q^{b_i}\parallel p_i$ for each $1\leq i\leq t$ and define
$$\text{index}_q({\bf{p}}):=\big|\big\{i|b_i=\text{max}\{b_1,\cdots,b_t\}\big\}\big|.$$
Obviously, if $q$ is not a factor of $p_i$ for any $1\leq i\leq t$, then each $b_i=0$, hence $\text{index}_q({\bf{p}})=t$.

{\begin{prop}\label{smooth2}
Keep the notations and assumptions in Proposition \ref{change of weight}. Then the following are equivalent:

\begin{itemize}
\item[$(1)$] there exists $H\leq t\bbL(\bf{p})$ such that $\mathbf{Y}$ is a smooth projective curve;

\item[$(2)$] for each prime number $q$, $\textup{index}_q(\mathbf{p})\geq 2$.
    \end{itemize}
\end{prop}}

{\begin{proof}
First, we assume the statement (1) is satisfied. Assume there exists a prime number $q$ such that $\textup{index}_q(\mathbf{p})=1$, i.e. there exists $k$ such that $b_k>b_i$ for each $i\neq k$. Denote $p=\text{l.c.m.}(p_1, p_2, \cdots, p_t)$. Hence
\begin{align}\label{q||p} q\nmid \frac{p}{p_k}\,\ \text{and}\,\ q\mid \frac{p}{p_i}\,\ \text{for each}\,\ i\neq k.
\end{align}

For any $0\neq \vec{y}=\sum_{i=1}^t a_i\vec{x}_i+a_0\vec{c}\in t\mathbb{L}(\bf{p})$ with $0\leq a_i< p_i$, then $p\vec{y}=\sum_{i=1}^t pa_i\vec{x}_i+pa_0\vec{c}=0$. Hence, $\sum_{i=1}^t \frac{a_ip^2}{p_i}+a_0p^2=0$. Let $q^b||p$, thus by \eqref{q||p}, $q^{b+1}\mid \frac{a_ip^2}{p_i}$ for each $i\neq k$ and $q^{b+1}\mid a_0p^2$. It follows that $q^{b+1}\mid \frac{a_kp^2}{p_k}$. Since $q^b||p$ and $q\nmid \frac{p}{p_k}$, we have $q|a_k$.

Therefore, for any subgroup $H$ of $t\mathbb{L}(\bf{p})$ generated by
$\vec{y}_i=\sum_{j=1}^t a_{ij}\vec{x}_j+a_{i}\vec{c}$ for $1\leq i\leq s$, we have $q\mid{\rm g.c.d.}(a_{1k}, a_{2k}, \cdots, a_{sk}, p_k)$. This contradicts Corollary \ref{smooth1}. Thus for each prime number $q$, $\textup{index}_q(\mathbf{p})\geq 2$.

Conversely, we assume (2) is satisfied. Fix $1\leq j\leq t$, write $p_j=\prod_{i=1}^m q_i^{d_i}$, where $q_i$'s are pairwise distinct primes. For any $1\leq i\leq m$, since $\text{index}_{q_i}(\mathbf{p})\geq 2$, there exists some $k\neq j$ such that $q_i^{d_i}\mid p_k$. Set $\vec{y}_i=\frac{p_j}{q^{d_i}_i}\vec{x}_j-\frac{p_k}{q^{d_i}_i}\vec{x}_k$ and $a_{ij}=\frac{p_j}{q^{d_i}_i}$. Then $\vec{y}_i\in t\mathbb{L}(\bf{p})$ and $q_i\nmid a_{ij}$. It follows that ${\rm g.c.d.}(a_{1j}, a_{2j}, \cdots, a_{mj}, p_j)=1$. By Corollary \ref{smooth1}, there exists $H=t\mathbb{L}(\bf{p})$ such that $\mathbf{Y}$ is a smooth projective curve. This finishes the proof.
\end{proof}

\begin{exm}\label{example for 2612} Let $\mathbf{p}=(2,6,12)$. Then $2\parallel 2,\ 2\parallel 6,\ 2^{2}\parallel 12$, hence $\textup{index}_2(\mathbf{p})=1$. According to Proposition \ref{smooth2}, there does not exist a subgroup $H\leq\bbL(\bf{p})$, such that
$$({\rm coh}\mbox{-}\mathbb{P}_{\mathbf k}^1(\mathbf{p}; {\boldsymbol\lambda}))^{H}\stackrel{\sim}\longrightarrow {\rm coh}\mbox{-}\bbY
$$ for some smooth projective curve $\bbY$.

\end{exm}

\begin{exm}\label{example for 4612} Let $\mathbf{p}=(4,6,12)$. Then $4=2^{2},\ 6=2\cdot3,\ 12=2^{2}\cdot3$, hence $\textup{index}_2(\mathbf{p})=2$ and $\textup{index}_3(\mathbf{p})=2$. By Proposition \ref{smooth2}, there exists $H\leq t\bbL(\bf{p})$ such that $$({\rm coh}\mbox{-}\mathbb{P}_{\mathbf k}^1(\mathbf{p}; {\boldsymbol\lambda}))^{H}\stackrel{\sim}\longrightarrow {\rm coh}\mbox{-}\bbY
$$ for some smooth projective curve $\bbY$. { In fact, we can take $H=t\mathbb{L}(\bf{p})$, then $\bbY$ is a smooth projective curve of genus 7 by Proposition \ref{change of weight}.}
\end{exm}}

\subsection{Relation to hyperelliptic curves of arbitrary genus}
For any $g>0$, let $a_1, a_2,\cdots, a_g$ be pairwise distinct elements in ${\bf{k}}\setminus\{0,1\}$.
Let ${\mathbf{E}}_g$ be the hyperelliptic curve of genus $g$ determined by the following equation: $$z^2=(y^2-x^2)(y^2-a_1x^2)\cdots(y^2-a_gx^2).$$

\begin{prop}\label{hyperelliptic}
Keep the notations as above. Let ${\bf{p}}=(\underbrace{2,\cdots, 2}_{g+3 \ \text{times}})$ and $\boldsymbol\lambda=\{\infty,0,1,a_1,\cdots,a_g\}$. Then there exists an equivalence
\begin{align}\label{hyper equiv}({\rm coh}\mbox{-}\mathbb{P}_{\mathbf k}^1(\mathbf{p}; {\boldsymbol\lambda}))^{H}\stackrel{\sim}\longrightarrow {\rm coh}\mbox{-}\mathbf{E}_g,\end{align}
where ${H=\langle \vec{x}_{1}-\vec{x}_{2},\,\delta_{g,even} \vec{x}_{2}+\sum_{i=3}^{g+3}(-1)^i \vec{x}_i \rangle.}$
\end{prop}

{\begin{proof}
Denote by $a_0=1$. Let $S=\mathbf k[X_1,X_2,\cdots,X_{g+3}]/I$, where $I=(X_i^2-X_2^2+a_{i-3}X_1^2|i=3,4,\cdots,g+3)$. Denote by $x_j=X_j+I$ for $1\leq j\leq g+3$. Then $S$ is an $\bbL({\bf{p}})$-graded algebra with $\textup{deg}(x_j)=\vec{x}_j$ for any $j$. Let $H'=\langle\vec{x}_1,\vec{x}_2,\sum_{i=3}^t\vec{x}_i\rangle$.
By easy calculation we obtain that the restriction subalgebra $S_{H'}$ is isomorphic to $$\mathbf k[u,v,w]/\big(w^2-(v^2-u^2)(v^2-a_{1}u^2)\cdots(v^2-a_{g}u^2)\big),$$ which can be viewed as an $H'$-graded algebra by defining $\deg(u)=\vec{x}_1,\deg(v)=\vec{x}_2,\deg(w)=\sum_{i=3}^t\vec{x}_i$. Since $H'$ is effective in the sense of \cite[Definition 6.5]{CCZ},
by \cite[Proposition 6.6]{CCZ} we have \begin{align}\label{L and H'}\text{qmod}^{\bbL(\bf{p})}S\stackrel{\sim}\longrightarrow\text{qmod}^{H'}S_{H'}.\end{align}

Observe that $H'=\mathbb{Z}\vec{x}_1\oplus H$, hence $H'/H\cong \mathbb{Z}$. We define a $\mathbb{Z}$-graded algebra $\overline{S_{H'}}$ as follows: as an ungraded algebra $\overline{S_{H'}}=S_{H'}$, while its homogeneous component is given by $(\overline{S_{H'}})_n=\oplus_{\vec{h}\in H}(S_{H'})_{n\vec{x}_1+\vec{h}}$ for each $n\in\mathbb{Z}$. By \cite[Proposition 5.2]{CCZ} we have an equivalence of categories
$$(\mod^{H'}S_{H'})^H\stackrel{\sim}\longrightarrow\mod^{\mathbb{Z}}\overline{S_{H'}},$$
which restricts to an equivalence $(\mod_0^{H'}S_{H'})^H\stackrel{\sim}\longrightarrow\mod_0^{\mathbb{Z}}\overline{S_{H'}}$ of full subcategories consisting of finite dimensional modules. Applying \cite[Corollary 4.4]{CCZ} we have an equivalence of categories\begin{align}\label{H and H'}(\text{qmod}^{H'}S_{H'})^H\stackrel{\sim}\longrightarrow\text{qmod}^{\mathbb{Z}}\overline{S_{H'}}.
\end{align} Combining with \eqref{L and H'} and \eqref{H and H'} we have $$(\text{qmod}^{\bbL(\bf{p})}S)^H\stackrel{\sim}\longrightarrow\text{qmod}^{\mathbb{Z}}\overline{S_{H'}}.$$
This proves \eqref{hyper equiv}.
\end{proof}}

\subsection{Relation to smooth projective curves of genus 2}

Recall that ${\bf p}=(p_1, p_2, \cdots, p_t)$ and $\bbL({\bf p})=\bbZ\vec{x}_1\oplus\cdots\oplus\bbZ\vec{x}_t/(p_1\vec{x}_1=p_2\vec{x}_2=\cdots = p_t\vec{x}_t:=\vec{c})$. Let $\sigma$ be a permutation on $\{1,2,\cdots,t\}$ and ${\bf p}^{\sigma}:=({p}_{\sigma(1)},{p}_{\sigma(2)},\cdots,{p}_{\sigma(t)})$. Denote by $\bbL({\bf p}^{\sigma}):=\bbZ\vec{y}_1\oplus\cdots\oplus\bbZ\vec{y}_t/({p}_{\sigma(1)}\vec{y}_1={p}_{\sigma(2)}\vec{y}_2=\cdots = {p}_{\sigma(t)}\vec{y}_t).$ Then there is a natural \emph{permutation isomorphism}
$$\pi_{\sigma}:\; \bbL({\bf p})\rightarrow\bbL({\bf p}^{\sigma});\; \vec{x}_i\mapsto\vec{y}_{\sigma^{-1}(i)}\;\ \textup{for}\;\ 1\leq i\leq t.$$

{\begin{prop}
Let $\mathbb{P}_{\mathbf k}^1(\mathbf{p};{\boldsymbol\lambda})$ be a weighted projective line and $H$ be a finite subgroup of $\mathbb{L}(\bf{p})$. Assume
$$({\rm coh}\mbox{-}\mathbb{P}_{\mathbf k}^1(\mathbf{p}; {\boldsymbol\lambda}))^{H}\stackrel{\sim}\longrightarrow {\rm coh}\mbox{-}\mathbb{Y}$$ for some smooth projective curve $\mathbb{Y}$ of genus 2. Then all the possibilities for $(\mathbf{p}, H)$ are classified in Table \ref{smooth of genus 2} (up to permutation isomorphisms).
\begin{table}[h]
\caption{Equivariant relations for genus $2$}
\begin{tabular}{|c|c|c|c|c|}
  \hline
 $\bf{p}$&\makecell*[c]{$H$}\\
\hline
 \makecell*[c]{$(2,5,10)$}&$t\bbL(2,5,10)$\\
  \hline
 \makecell*[c]{$(2,6,6)$}&$t\bbL(2,6,6)$\\
  \hline
  \makecell*[c]{$(2,8,8)$}&$\langle\vec{x}_1+\vec{x}_2+3\vec{x}_3-\vec{c}\rangle$\\
  \hline
  \makecell*[c]{$(3,6,6)$}&$\langle\vec{x}_1+5\vec{x}_2+5\vec{x}_3-2\vec{c}\rangle$\\
  \hline
  \makecell*[c]{$(5,5,5)$}&$\langle\vec{x}_1+\vec{x}_2+3\vec{x}_3-\vec{c}\rangle$\\
  \hline
  \makecell*[c]{$(2,2,3,3)$}&$t\bbL(2,2,3,3)$\\
  \hline
  \makecell*[c]{$(2,2,4,4)$}&$\langle\vec{x}_1+\vec{x}_2+\vec{x}_3+3\vec{x}_4-2\vec{c}\rangle$\\
  \hline
  \makecell*[c]{$(3,3,3,3)$}&$\langle\vec{x}_1+\vec{x}_2+2\vec{x}_3+2\vec{x}_4-2\vec{c}\rangle$\\
  \hline
  \makecell*[c]{$(2,2,2,2,2)$}&$\langle\sum_{i=1}^4 \vec{x}_i-2\vec{c}, \vec{x}_4+\vec{x}_5- \vec{c}\rangle$\\
  \hline
  \makecell*[c]{$(2,2,2,2,2,2)$}&$\langle\sum_{i=1}^6 \vec{x}_i-3\vec{c}\rangle$\\
  \hline
\end{tabular}
\label{smooth of genus 2}
\end{table}\end{prop}}

{\begin{proof}
Let $\mathbf{p}=(p_1,\cdots,p_t)$ and denote by $|H|=n$. Up to permutation isomorphisms we can assume $2\leq p_1\leq p_2\leq\cdots\leq p_t$ without loss of generality. Since $\mathbb{Y}$ is smooth, by Corollary \ref{smooth1} and Proposition \ref{change of weight} we have $\frac{n}{p_t}=n_t\geq 1$, i.e. $p_t\leq n$.

According to \cite[Theorem 2]{Lenzing2017}, we have
$$2-\sum_{i=1}^t (1-\frac{1}{p_i})=\chi_{\mathbb{P}_{\mathbf k}^1(\mathbf{p};{\boldsymbol\lambda})}=\frac{\chi_{\mathbb{Y}}}{n}=\frac{2(1-g_\mathbb{Y})}{n}=-\frac{2}{n}.$$
It follows that
\begin{equation}\label{genus 2 euler char} t-2=\sum_{i=1}^t \frac{1}{p_i}+\frac{2}{n}\leq \frac{t+2}{p_1}.\end{equation}
Then $p_1\geq 2$ implies $t\leq 6$. Moreover,
by Theorem \ref{trichotomy} we know that $\mathbb{P}_{\mathbf k}^1(\mathbf{p};{\boldsymbol\lambda})$ is of wild type, hence $t\geq 3$.

For $t=6$, by \eqref{genus 2 euler char} we have $4=\sum_{i=1}^6 \frac{1}{p_i}+\frac{2}{n}\leq \frac{8}{p_1}\leq 4$, which forces $\mathbf{p}=(2,2,2,2,2,2)$ and $n=2$. Assume $H=\langle \vec{y}\rangle$ with normal form $\vec{y}=\sum\limits_{i=1}^6 a_i\vec{x}_i+a\vec{c}$.
According to Corollary \ref{smooth1}, we have $\text{g.c.d.}(a_i, p_i)=1$ for $1\leq i\leq 6$, hence $a_i=1$ for any $i$. Now $n=2$ implies $a=-3$, that is,
$H=\langle\sum_{i=1}^6 \vec{x}_i-3\vec{c}\rangle$.

For $t=5$, \eqref{genus 2 euler char} reads $3=\sum_{i=1}^5 \frac{1}{p_i}+\frac{2}{n}\leq \frac{7}{p_1}.$ By similar arguments as above, we get $p_i=2$ for $1\leq i\leq 4$ and $p_5\leq 3$. If $p_5=3$, we have $\textup{index}_3(\mathbf{p})=1$, contradicting to Proposition \ref{smooth2}, hence $p_5=2$ and then $n=4$ follows. Note that each torsion element in $t\bbL(2,2,2,2,2)$ has order 2. Hence $H$ has the form $H=\langle \vec{y}_1, \vec{y}_2\rangle$ with normal forms $\vec{y}_k=\sum_{i=1}^5 a_{ki}\vec{x}_i+a_k\vec{c}$ for $k=1,2$. According to Corollary \ref{smooth1}, we have $\text{g.c.d.}(a_{1i}, a_{2i}, p_i)=1$ for $1\leq i\leq 5$, hence $(a_{1i}, a_{2i})\neq (0,0)$ for any $i$. By the explicit description of torsion group $t\bbL(2,2,2,2,2)$, we obtain that $H=\langle\sum_{i=1}^4 \vec{x}_i-2\vec{c}, \vec{x}_4+\vec{x}_5- \vec{c} \rangle$, up to permutation isomorphisms.

{For $t=4$, by \eqref{genus 2 euler char} we have $2=\sum_{i=1}^4 \frac{1}{p_i}+\frac{2}{n}\leq \frac{6}{p_1},$ hence $p_1=2$ or $3$. Combining with Proposition \ref{smooth2}, by using $\textup{index}_q(\mathbf{p})\geq 2$ for all prime number $q$, we obtain that there are only the following three cases for $\mathbf{p}$:
$$(i)\, (2,2,3,3),\quad\; (ii)\, (2,2,4,4),\quad\ (iii)\, (3,3,3,3).$$
It is easy to calculate that $n=6,\; 4,\; 3$ respectively. Now by using Corollary \ref{smooth1}, we can derive the forms of $H$ as follows respectively:
$$(i)\, t\bbL(2,2,3,3),\; (ii)\, \langle\vec{x}_1+\vec{x}_2+\vec{x}_3+3\vec{x}_4-2\vec{c}\rangle,\; (iii)\,  \langle\vec{x}_1+\vec{x}_2+2\vec{x}_3+2\vec{x}_4-2\vec{c}\rangle.$$}

For $t=3$, \eqref{genus 2 euler char} reads $1=\sum_{i=1}^3 \frac{1}{p_i}+\frac{2}{n}\leq \frac{5}{p_1},$ hence $p_1\leq 5$. Combining with Proposition \ref{smooth2}, by using $\textup{index}_q(\mathbf{p})\geq 2$ for all prime number $q$, we obtain that there are only the following five cases for $\mathbf{p}$:
$$(i)\, (2,5,10),\quad\; (ii)\, (2,6,6),\quad\; (iii)\, (2,8,8),\quad\; (iv)\, (3,6,6),\quad\; (v)\, (5,5,5).$$
It is easy to see that $n=10,\; 12,\; 8,\; 6,\; 5$ respectively. According to Corollary \ref{smooth1}, we can derive the forms of $H$ as follows respectively:
$$\begin{array}{llllll}
(i)\, t\bbL(2,5,10), &(ii)\, t\bbL(2,6,6),&(iii)\, \langle\vec{x}_1+\vec{x}_2+3\vec{x}_3-\vec{c}\rangle,\\
(iv)\, \langle\vec{x}_1+5\vec{x}_2+5\vec{x}_3-2\vec{c}\rangle,& (v)\, \langle\vec{x}_1+\vec{x}_2+3\vec{x}_3-\vec{c}\rangle.&
\end{array}$$
To sum up, we obtain Table \ref{smooth of genus 2}.
\end{proof}}

\subsection{Relation to Arnold's exceptional unimodal singularities}

Arnold's exceptional unimodal singularities arise in the theory of singularities of differentiable
maps \cite{ArnoldGusejnZadeVarchenko1985}, which are related to the
mirror symmetry of K3 surfaces (see e.g. \cite{Dolgachev1996}), and have close relationship to
automorphic forms and Fuchsian singularities (see e.g. \cite{Lenzing1994, LenzingdelaPena2011,Dolgacev1975,Ebeling2003,Wagreich1980}).
There are 14 kinds of weighted projective lines associated to Arnold's 14 exceptional unimodal singularities.
For these 14 cases, we have the following result.

\begin{prop}
Let $\mathbb{P}_{\mathbf k}^1(\mathbf{p})$'s be the 14 weighted projective lines associated to Arnold's 14 exceptional unimodal singularities. Let $H$ be a finite subgroup of $\mathbb{L}(\bf{p})$. Then all the equivalences of the form $$({\rm coh}\mbox{-}\mathbb{P}_{\mathbf k}^1(\mathbf{p}))^{H}\stackrel{\sim}\longrightarrow {\rm coh}\mbox{-}\mathbf{Y}$$ are classified in Table \ref{Gorenstein parameter -1} (up to permutation isomorphisms).

\begin{table}[h]
\caption{Equivariant relations for Arnold's exceptional unimodal singularities}
\begin{tabular}{|c|c|c|c|c|c|}
  \hline
 \makecell*[c]{$\mathbb{P}_{\mathbf k}^1(p_1,p_2,p_3)$}&$H$&$\mathbf{Y}=\bbY(q_1,q_2,\cdots, q_s)$&$g_{\bbY}$\\
\hline
  \makecell*[c]{$\mathbb{P}_{\mathbf k}^1(2,3,7)$}&$$&&\\
  \hline
  \makecell*[c]{$\mathbb{P}_{\mathbf k}^1(2,3,8)$}&$\langle\vec{x}_1-4\vec{x}_3\rangle$&$\bbY(3,3,4)$&0\\
  \hline
  \makecell*[c]{$\mathbb{P}_{\mathbf k}^1(2,3,9)$}&$\langle\vec{x}_2-3\vec{x}_3\rangle$&$\bbY(2,2,2,3)$&0\\
  \hline
  \makecell*[c]{$\mathbb{P}_{\mathbf k}^1(2,4,5)$}&$\langle\vec{x}_1-2\vec{x}_2\rangle$&$\bbY(2,5,5)$&0\\
  \hline
  \multirow{5}*{$\mathbb{P}_{\mathbf k}^1(2,4,6)$}&\makecell*[c]{$\langle\vec{x}_1-2\vec{x}_2\rangle$}&$\bbY(2,6,6)$&0\\
  \cline{2-4}
  &\makecell*[c]{$\langle\vec{x}_1-3\vec{x}_3\rangle$}&$\bbY(3,4,4)$&0\\
  \cline{2-4}
  &\makecell*[c]{$\langle2\vec{x}_2-3\vec{x}_3\rangle$}&$\bbY(2,2,2,3)$&0\\
  \cline{2-4}
  &\makecell*[c]{$t\bbL(2,4,6)$}&$\bbY(2,2,3,3)$&0\\
  \hline
  \makecell*[c]{$\mathbb{P}_{\mathbf k}^1(2,4,7)$}&$\langle\vec{x}_1-2\vec{x}_2\rangle$&$\bbY(2,7,7)$&0\\
  \hline
  \makecell*[c]{$\mathbb{P}_{\mathbf k}^1(2,5,5)$}&$\langle\vec{x}_2-\vec{x}_3\rangle$&$\bbY(2,2,2,2,2)$&0\\
  \hline
  \makecell*[c]{$\mathbb{P}_{\mathbf k}^1(2,5,6)$}&$\langle\vec{x}_1-3\vec{x}_3\rangle$&$\bbY(3,5,5)$&0\\
  \hline
  \makecell*[c]{$\mathbb{P}_{\mathbf k}^1(3,3,4)$}&$\langle\vec{x}_1-\vec{x}_2\rangle$&$\bbY(4,4,4)$&0\\
  \hline
  \makecell*[c]{$\mathbb{P}_{\mathbf k}^1(3,3,5)$}&$\langle\vec{x}_1-\vec{x}_2\rangle$&$\bbY(5,5,5)$&0\\
  \hline
  \multirow{5}*{$\mathbb{P}_{\mathbf k}^1(3,3,6)$}&$\langle\vec{x}_1-\vec{x}_2\rangle$&\makecell*[c]{$\bbY(6,6,6)$}&0\\
  \cline{2-4}
  &$\langle\vec{x}_2-2\vec{x}_3\rangle$&\makecell*[c]{$\bbY(2,3,3,3)$}&0\\
  \cline{2-4}
  &$\langle\vec{x}_1+\vec{x}_2-4\vec{x}_3\rangle$&\makecell*[c]{$\bbY(2)$}&1\\
  \cline{2-4}
  &$t\bbL(3,3,6)$&\makecell*[c]{$\bbY(2,2,2)$}&1\\
  \hline
  \multirow{2}*{$\mathbb{P}_{\mathbf k}^1(3,4,4)$}&$\langle2\vec{x}_2-2\vec{x}_3\rangle$&\makecell*[c]{$\bbY(2,2,3,3)$}&0\\
  \cline{2-4}
  &$\langle\vec{x}_2-\vec{x}_3\rangle$&\makecell*[c]{$\bbY(3,3,3,3)$}&0\\
  \hline
  \makecell*[c]{$\mathbb{P}_{\mathbf k}^1(3,4,5)$}&$$&&\\
  \hline
  \multirow{8}*{$\mathbb{P}_{\mathbf k}^1(4,4,4)$}&$\langle2\vec{x}_1-2\vec{x}_2\rangle$&\makecell*[c]{$\bbY(2,2,4,4)$}&0\\
  \cline{2-4}
  &$\langle\vec{x}_1-\vec{x}_2\rangle$&\makecell*[c]{$\bbY(4,4,4,4)$}&0\\
  \cline{2-4}
  &$\langle\vec{x}_1+\vec{x}_2-2\vec{x}_3\rangle$&\makecell*[c]{$\bbY(2,2)$}&1\\
  \cline{2-4}
  &$\langle2\vec{x}_1-2\vec{x}_2,2\vec{x}_2-2\vec{x}_3\rangle$&\makecell*[c]{$\bbY(2,2,2,2,2,2)$}&0\\
  \cline{2-4}
  &$\langle\vec{x}_1-\vec{x}_2,2\vec{x}_2-2\vec{x}_3\rangle$&\makecell*[c]{$\bbY(2,2,2,2)$}&1\\
  \cline{2-4}
  &$t\bbL(4,4,4)$&\makecell*[c]{$\bbY$}&3\\
  \hline
\end{tabular}
\label{Gorenstein parameter -1}
\end{table}
\end{prop}

\begin{proof}
The weight types $\bf{p}$'s of the weighted projective lines $\mathbb{P}_{\mathbf k}^1(\mathbf{p})$'s associated to Arnold's 14 exceptional unimodal singularities are well-known; see e.g. \cite[Table 1]{Wagreich1980}. For such a weight type $\bf{p}$, it is easy to calculate the subgroup $H$ of $t\mathbb{L}(\bf{p})$ by the explicit description of the torsion group $t\mathbb{L}(\bf{p})$ in Lemma \ref{torsiongp}. Then for each $H$, by Theorem \ref{trichotomy} we have an equivalence \begin{equation}({\rm coh}\mbox{-}\mathbb{P}_{\mathbf k}^1(\mathbf{p}; {\boldsymbol\lambda}))^{H}\stackrel{\sim}\longrightarrow {\rm coh}\mbox{-}\mathbf{Y}
 \end{equation} for some weighted projective curve $\mathbf{Y}=\mathbb{Y}(\mathbf{q};{\boldsymbol\mu})$, where the weight type ${\bf q}$ and the genus $g_{\bbY}$ can be calculated by Proposition \ref{change of weight}. This finishes the proof.
\end{proof}

\noindent {\bf Acknowledgements.}\quad
The authors would like to thank Jianmin Chen and Xiao-wu Chen for their useful discussions.
The second author would like to thank Helmut Lenzing for explaining the relations between geometric actions and degree-shift actions on weighted projective lines.
This work was supported by the National Natural Science Foundation of China (No. 11801473).

\bibliographystyle{plain}

\end{document}